\documentclass{article}
\usepackage{amsmath,graphicx,url}
\usepackage{amsthm}
\usepackage[textwidth=16cm]{geometry}

\usepackage[utf8]{inputenc}
\usepackage[draft,inline,nomargin]{fixme}

\usepackage{caption}
\usepackage{subcaption}
\usepackage{tikz}
\usepackage{pgfplots}
\graphicspath{{./figures/}}

\newcommand{\RR}{\mathbf{R}}
\newcommand{\norm}[1]{\|#1\|}
\newcommand{\abs}[1]{|#1|}
\DeclareMathOperator{\sign}{sign}
\DeclareMathOperator*{\argmin}{argmin}

\newtheorem{thm}{Theorem}[section]

\theoremstyle{definition}

\newtheorem{rem}{Remark}

\title{A sparse Kaczmarz solver and a linearized Bregman method for online compressed sensing}
\author{Dirk Lorenz\and Stephan Wenger \and Frank Sch\"opfer \and Marcus Magnor}
\begin{document}
%
\maketitle
\begin{abstract}
  An algorithmic framework to compute sparse or minimal-TV solutions
  of linear systems is proposed. The framework includes both the
  Kaczmarz method and the linearized Bregman method as special cases
  and also several new methods such as a sparse Kaczmarz solver.

  The algorithmic framework has a variety of applications and is
  especially useful for problems in which the linear measurements are
  slow and expensive to obtain. We present examples for online
  compressed sensing, TV tomographic reconstruction and radio
  interferometry.
\end{abstract}

\section{Introduction}
\label{sec:intro}

Sparse solutions of linear systems play a vital role in several active
fields and form the backbone of the theory known as compressed
sensing~\cite{candes2006compressive,FR13}. A prominent approach to compute such
sparse solutions in the underdetermined case is to compute solutions
with minimal $\ell^1$-norm and this is known as Basis
Pursuit~\cite{chen1998basispursuit}. There exists a large body of literature to compute
solutions to this problem and we only refer to the recent
review~\cite{lorenz2011ISAL1}. In this work we propose a framework that addresses two fundamental problems: 1) The problem may be too large to fit into the
computer's memory and hence, only parts of the whole problem can be
processed at once. 2) The measurement process may be very slow, i.e.~it
takes considerable time to obtain a new row of the measurement matrix
and the corresponding entry in the right hand side.
Our algorithmic framework to compute sparse solutions uses
(in the extreme case) only a
single row of the measurement matrix in each step. The
framework is flexible enough to include ``block processing'' of the
matrix and also allows for generalization to sparse solutions in
dictionaries (also known as $\ell^1$-analysis
minimization~\cite{VaPeDoFa13,H13}) and total variation
minimization~\cite{chambolle1997tvmin}. Our work can be seen as a combination of the
Kaczmarz-approach~\cite{K37} (also known as ART~\cite{GBH70}) and the
linearized Bregman method~\cite{OMDY10} and we follow the framework laid
out in~\cite{lorenz2013linbreg}. Note that a different and heuristic
sparse Kaczmarz solver have been proposed in~\cite{MY13} and a methodology for online compressed sensing based on homotopy was proposed in~\cite{garrigues2008homotopy,SR10}.

\section{The sparse Kaczmarz solver and the linearized Bregman method}
\label{sec:sparse-kaczmarz-solver}

\subsection{The Kaczmarz solver}
\label{sec:kaczmarz}

We consider a matrix $A\in\RR^{m\times n}$ with $n>m$ with rows
$a_k^T$ ($k=1,\dots m$), a vector $b\in\RR^m$ and aim to find
solutions to the underdetermined system of equations $Ax=b$. A
classical iterative method that only uses a single row $a_k^T$ in each
step is Kaczmarz's method~\cite{K37}. It consists of iterative projections onto
the hyperplanes $H_k = \{x\ :\ a_k^T x = b_k\}$,
i.e. in iteration $k$ we choose an index $r(k)$ and perform $x^{k+1} =
P_{H_{r(k)}}(x^k) = x^k - \frac{a_{r(k)}^Tx^k -
  b_{r(k)}}{\norm{a_k}^2}a_{r(k)}$. The \emph{control sequence} could be
simply cyclic ($r(k) = (k \mod m) + 1$) but also randomized sequences
are of use (cf.~\cite{strohmer2009randomized} where a convergence rate is
proven when one samples the rows with probability equal to their
squared norm). Convergence can be ensured, if the control sequence
picks up any index infinitely often~\cite{BB97} and we call these
control sequences \emph{admissible}. The Kaczmarz method converges for
any system $Ax=b$ that has a solution and moreover, it converges
towards the least squares solution.
Interestingly, a very small change in the algorithm makes it converge
to a sparse solution:
\begin{thm}[Convergence of the sparse Kaczmarz method]
  Assume that $Ax=b$ has a
  solution, let $\lambda>0$, $r$ be an admissible control sequence and
  denote by
  $
  S_\lambda(x) = \max(\abs{x}-\lambda,0)\sign(x)
  $
  the soft shrinkage function (applied component-wise to a
  vector). Then for $z^0=x^0=0$ the iteration
  \begin{equation}\label{eq:sparse_kaczmarz}
    \begin{split}
      z^{k+1} & = z^k - \frac{a_{r(k)}^Tx^k -
        b_{r(k)}}{\norm{a_{r(k)}}^2}a_{r(k)}\\
      x^{k+1} & = S_\lambda(z^{k+1})
    \end{split}
  \end{equation}
  converges to a solution of
  \begin{equation}
    \label{eq:sparse_solution}
    \min_x \lambda\norm{x}_1 + \tfrac12\norm{x}^2,\quad\text{s.t.}\quad Ax=b.
  \end{equation}
\end{thm}
\begin{proof}
  The theorem follows from the theory developed
  in~\cite{lorenz2013linbreg} as the sparse Kaczmarz method is a
  special instance of the BPSFP (Bregman projections for split
  feasibility problems): Consider the problem as a feasibility problem
  ``find $x\in \bigcap_{k=1}^m H_k$'' (with the hyperplanes $H_k$
  defined above) and define $f(x) = \lambda\norm{x}_1 +
  \tfrac12\norm{x}^2$. Then the sparse Kaczmarz method is a special
  instance of BPSFP with Bregman projections w.r.t. $f$ and inexact
  stepsizes according to~\cite[Theorem 2.8]{lorenz2013linbreg}.
\end{proof}

\begin{rem}
  \label{rem:exact_step}
  The sparse Kaczmarz method converges much faster if an \emph{exact
    stepsize}, as described in~\cite[Section 2.5.2]{lorenz2013linbreg}, is used. There one updates $z$ as $z^{k+1} = z^k-t_k
  a_{r(k)}$ with exact stepsize $t_k$ given as the solution of
  the problem
  \[
  t_k = \argmin_{t\in\RR} f^*(z^k - t\, a_{r(k)}) + t b_{r(k)}.
  \]
  Here $f^*$ is the convex dual of $f$, i.e. $f^*(z) =
  \tfrac12 \norm{S_\lambda(z)}^2$, in other words, $t_k$ is a solution of the
  piecewise linear equation
  \[
  a_{r(k)}^T\,S_\lambda(z^k - t\, a_{r(k)}) = b_{r(k)}.
  \]
\end{rem}

\subsection{From sparse block-Kaczmarz to the linearized Bregman method}
\label{sec:block-linbreg}

The flexibility of the BPSFP-framework allows to devise a blockwise
method as follows: Split the matrix $A\in\RR^{m\times n}$ into $L$
blocks $A_l\in\RR^{m_l\times n}$ (with $m_1 + \dots + m_L = m$) and
set the corresponding part of the right hand side $b\in\RR^m$ to
$b_l\in\RR^{m_l}$. Then each equation $A_l x = b_l$ has an affine
linear space $L_l$ of solutions and the solution space of $Ax=b$ is
$\bigcap_{l=1}^L L_l$.  The basic block-Kaczmarz method (with control
sequence $r$) then reads as
\[
x^{k+1} = x^k - A_{r(k)}^\dag (A_{r(k)} x^k - b_{r(k)})
\]
(i.e. the current iterate is orthogonally projected onto $L_{r(k)}$ by the
pseudo-inverse of $A_{r(k)}$). A similar treatment
would be possible for the sparse Kaczmarz method, but the
corresponding Bregman projection onto $L_{r(k)}$ is not straight
forward. A simpler method is based on the following idea: For each
block compute  a hyperplane that separates the current iterate
from $L_{r(k)}$ and then calculate the Bregman projection onto this
hyperplane. The iteration is
\begin{equation}\label{eq:iterations}
  \begin{split}
    t_k  & = \frac{\norm{A_{r(k)}x^k - b_{r(k)}}^2}{\norm{A_{r(k)}^T(A_{r(k))}x^k - b_{r(k)})}^2}\\
    z^{k+1} & = z^k - t_k A_{r(k)}^T(A_{r(k)}x^k - b_{r(k)})\\
    x^{k+1} & = S_\lambda(z^{k+1})
  \end{split}
\end{equation}
and converges to a solution of~(\ref{eq:sparse_solution})
(cf.~\cite[Corollary 2.9]{lorenz2013linbreg}).
\begin{rem}
  \label{rem:exact_step_block}
  The stepsize $t_k$ was called ``dynamic'' stepsize
  in~\cite{lorenz2013linbreg}. Another (usually worse) alternative is
  $t_k = \norm{A_{r(k)}}^{-2}$ but also an exact stepsize similar
  to Remark~\ref{rem:exact_step} is possible 
  (cf.~\cite[Algorithm 1]{lorenz2013linbreg}) and usually leads to
  faster convergence.
\end{rem}

In the extreme case of just one block ($L=1$), the resulting iteration
reads as
\begin{equation}\label{eq:linbreg}
  \begin{split}
    z^{k+1} & = z^k - t_k A^T(Ax^k - b_{r(k)})\\
    x^{k+1} & = S_\lambda(z^{k+1})
  \end{split}
\end{equation}
and is precisely the \emph{linearized Bregman method} from~\cite{OMDY10}
(up to the stepsize $t_k$ which is set constant $t= \norm{A}^{-2}$
there).

\section{Online Compressed Sensing}
\label{sec:online-cs}

As an illustration of the potential of the proposed method, we take
the following compressed sensing scenario: Assume that a sparse quantity
$x\in \RR^n$ can be measured by linear measurements, i.e. that one can
generate numbers $b_k = a_k^T x$ for some vectors $a_k\in\RR^n$, and
for simplicity we assume that the vectors $a_k$ contain Gaussian
randomly distributed entries. Moreover, assume that it is both costly
and time-consuming to take one measurement. In the classical
compressed sensing scenario one would estimate the sparsity of the
unknown solution (i.e. the number of nonzeros in $x$), then calculate
the number $m$ of measurements that is needed to guarantee that the
solution of $\min_x \norm{x}_1$, $a_k^T x = b_k$, $k=1,\dots, m$ is
the exact sparse solution and solve the optimization problem (or with
$\norm{x}_1$ replaced by $\lambda\norm{x}_1 + \tfrac12\norm{x}^2$
for some large enough $\lambda$, cf.~\cite{Yin10,schoepfer2012exact}). In our new
framework, we can start solving
\[
\min_x \lambda\norm{x}_1 + \tfrac12\norm{x}^2\quad\text{s.t.}\quad a_k^Tx = b_k,\ k=1,\dots,l
\]
as soon as the first $l$ measurements have been taken. There are at
least two different possibilities to do so:
\begin{enumerate}
\item \textbf{Increasing cycle sparse Kaczmarz:} Perform sweeps of the
  sparse Kaczmarz method~(\ref{eq:sparse_kaczmarz}) with $r(k) = (k \mod l) +1$.
\item \textbf{Increasing linearized Bregman method:} Perform
  linearized Bregman iterations~(\ref{eq:linbreg}) with $A^l =
  [a_1,\dots a_l]^T$ and $b^l = [b_1,\dots,b_l]^T$.
\end{enumerate}
To illustrate the performance of both methods we generated a vector $x\in\RR^n$ with $n=1500$ and 20 non-zero entries, initialized with a matrix $A$
consisting of a single row. Then we started the increasing cycle sparse
Kaczmarz and the increasing linearized Bregman method, added a row to
$A$ every 0.1~seconds and plotted the relative residual $\norm{A_lx^k
  - b_l}/\norm{b_l}$ and the relative reconstruction error $\norm{x^k
  - x^\dag}/\norm{x^\dag}$ against time in
Figure~\ref{fig:cs:res-error}. One clearly observes an interesting
phenomenon: The residual decreases as expected and in the beginning,
it shoots up to a large value as soon as a new line is added to the
linear system; but from some point on this is not true anymore and the
residual stays small and most interestingly: This happens precisely
when the reconstruction error drops down drastically, indicating that
the true sparse solution has been found. This effect can be exploited
to stop doing measurements much earlier than indicated by a
precomputed number $m$ by just observing the evolution of the residual.

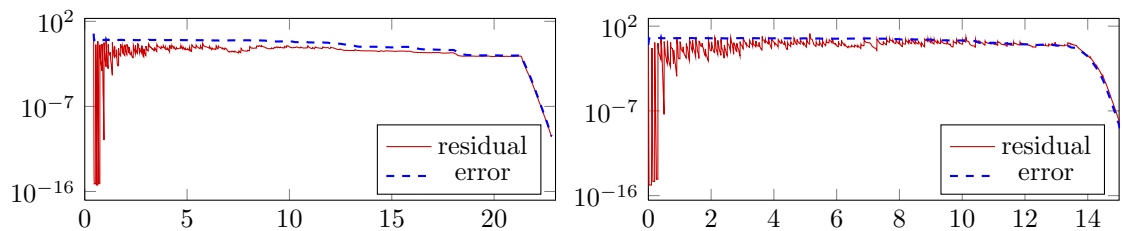
\begin{figure}[ht]
	\begin{tikzpicture}
		\begin{axis}[
				width=0.49\linewidth,
				height=0.25\linewidth,
				ymode=log,
				xmin=0,
				xmax=23,
				scaled ticks=false,
				legend entries={residual,error},
				legend pos = south east
			]
			\addplot[style=solid,red!80!black] table {data/cs_example/residualK.dat};
			\addplot[style=dashed,blue,thick] table {data/cs_example/errorxK.dat};
		\end{axis}
	\end{tikzpicture}\
	\begin{tikzpicture}
		\begin{axis}[
				width=.49\linewidth,
				height=0.25\linewidth,
				ymode=log,
				xmin=0,
				xmax=15,
				scaled ticks=false,
				legend entries={residual,error},
				legend pos = south east
			]
			\addplot[style=solid,red!80!black] table {data/cs_example/residualLB.dat};
			\addplot[style=dashed,blue,thick] table {data/cs_example/errorxLB.dat};
		\end{axis}
	\end{tikzpicture}
	\caption{Illustration of the online compressed sensing
          scenario. Top: increasing cycle sparse Kaczmarz method
          against the computation time (final number of measurements is $m=196$).
          Bottom: increasing linearized Bregman
          method (final number of measurement is $m=150$).}
	\label{fig:cs:res-error}
\end{figure}

\section{TV tomography}
\label{sec:comp-example}

To illustrate the flexibility of the BPSFP framework, we show a
possibility to handle total variation minimization problems in a
Kaczmarz style. The Kaczmarz method has proven to be an efficient tool in computerized tomography. It is
used to compute least squares solutions in the case of an
overdetermined system $Ax=b$,
i.e. minimizers of $\norm{Ax-b}^2$. To lower the dose of radiation
it is desirable to reduce the number of X-ray projections.
If the resulting linear system is
underdetermined, further regularization is needed to obtain meaningful
solutions. We consider a phantom of $128\times 128$ pixels and 3128 measurements (i.e. approximately five times  undersampling)\footnote{We use the AIRtools toolbox~\cite{HS12} to obtain the tomography matrix.}. total variation regularization, i.e. we are
interested in the solution of
\begin{equation}
  \label{eq:tv-tomo}
  \min_u \norm{|\nabla u|}_1\quad\text{s.t}\quad Au = b
\end{equation}
(where $|\nabla u|$ stands for the Euclidean norm of the gradient of
$u$, applied pointwise).  As such the problem does not fit directly
into our algorithmic framework and we propose to introduce an
auxiliary variable $p$, add a quadratic regularization and reformulate~\eqref{eq:tv-tomo} as
\begin{equation}
  \label{eq:tv-tomo-aux}
  \begin{split}
    \min_{u,p}\lambda\norm{|p|}_1 + \tfrac12\Big(\norm{u}^2 + \norm{p}^2\Big)\quad\text{s.t}\quad  Au &= b,\\
    \nabla u &= p.
  \end{split}
\end{equation}
Now we treat both constraints separately: The constraints $Au=b$ are
treated as single hyperplane constraints $a_k^T u = b_k$. This leads
to well known Kaczmarz steps for the $u$ variable: $u^{k+1} = u^k -
\frac{a_{r(k)}^T u^k - b_{r(k)}}{\norm{a_{r(k)}}^2}a_{r(k)}$. For the
constraint $B[u,p]^T = \nabla u - p =0$ we perform linearized Bregman
steps (as the application of the full operator $B= [\nabla\ -I]$ is
very cheap). Since the functional $f(p) = \lambda\norm{|p|}_1 +
\tfrac12\norm{p}^2$ does not allow for a simple calculation of exact
stepsizes (cf. Remark~\ref{rem:exact_step_block}), we use the dynamic
stepsize as in~\eqref{eq:iterations} and get the linearized Bregman
steps
\begin{equation}
  \label{eq:2}
  \begin{split}
    w^k & = B [u,p]^T = \nabla u^k - p^k,\qquad
    t_k  = \frac{\norm{w^k}^2}{\norm{B^Tw^k}^2}\\
    \begin{bmatrix}
      v^{k+1}\\ q^{k+1}
    \end{bmatrix}
    & = 
    \begin{bmatrix}
      v^{k}\\ q^{k}
    \end{bmatrix} - t_k B^T w^k,\\
    u^{k+1} & = v^{k+1}\\
    p^{k+1} & = S^2_\lambda(q^{k+1})
  \end{split}
\end{equation}
with the two-dimensional shrinkage function $S_\lambda^2:\RR^2 \to
\RR^2$, $S_\lambda^2(x) = \max(|x|-\lambda,0)\tfrac{x}{|x|}$.  The
proposed framework allows to perform steps for either constraint in an
arbitrary order. We show the results for the TV-Kaczmarz method where we perform either one linearized Bregman (LB) step per Kaczmarz sweep or 100 LB steps per sweep, cf.~Figures~\ref{fig:tomography} and~\ref{fig:phantom}.

\begin{figure}[ht]
	\begin{tikzpicture}
		\begin{axis}[
				width=0.48\linewidth,
				height=0.25\linewidth,
				ymode=log,
				xmin=0,
				xmax=500,
				scaled ticks=false,
				legend entries={1 LB step, 100 LB steps},
				legend pos = north east
			]
			\addplot[style=solid,red!80!black] table {data/tomography/res1_1.dat};
                        \addplot[thick,style=dotted,green!60!black] table {data/tomography/res1_100.dat};
		\end{axis}
	\end{tikzpicture}
	\begin{tikzpicture}
		\begin{axis}[
				width=0.48\linewidth,
				height=0.25\linewidth,
				ymode=log,
				xmin=0,
				xmax=500,
				scaled ticks=false,
				legend entries={1 LB step, 100 LB steps},
				legend pos = north east
			]
			\addplot[style=solid,red!80!black] table {data/tomography/res2_1.dat};
                        \addplot[thick,style=dotted,green!60!black] table {data/tomography/res2_100.dat};
                        
		\end{axis}
	\end{tikzpicture}
	\caption{Illustration for the TV-Kaczmarz solver. Top:
          Relative resiudal $\norm{Au^k - b}/\norm{b}$, bottom:
          Residual $\norm{\nabla u - p}$. Red solid: One linearized
          Bregman step per Kaczmarz sweep, green dotted: 100
          linearized Bregman steps per Kaczmarz sweep.}
	\label{fig:tomography}
\end{figure}
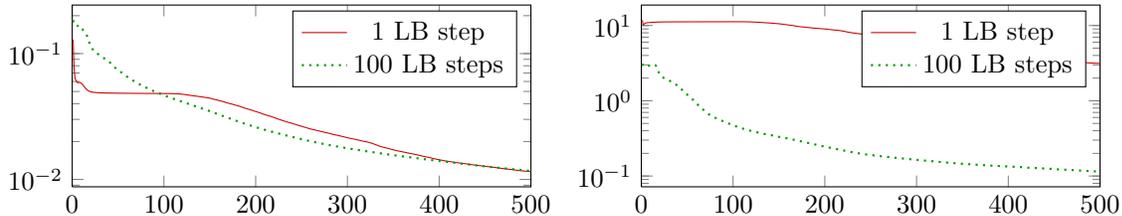

\begin{figure}[ht]
  \begin{center}
    \begin{subfigure}{0.25\linewidth}
      \includegraphics[width=\textwidth]{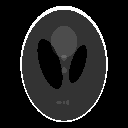}
      \caption{Phantom.\vphantom{Xg}}
      \label{fig:phantom-orig}
    \end{subfigure}\hspace{0.02\linewidth}%
    \begin{subfigure}{0.25\linewidth}
      \includegraphics[width=\textwidth]{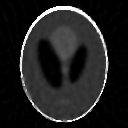}
      \caption{1 LB step per sweep.}
      \label{fig:phantom-rec1}
    \end{subfigure}\hspace{0.02\linewidth}%
    \begin{subfigure}{0.25\linewidth}
      \includegraphics[width=\textwidth]{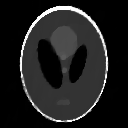}
      \caption{100 LB steps per sweep.}
      \label{fig:phantom-rec100}
    \end{subfigure}
  \end{center}
	\caption{TV reconstructions with the TV-Kaczmarz solver.}
        \label{fig:phantom}
\end{figure}

\section{Radio interferometry}
\label{sec:radio-inter}

In radio interferometry, multiple radio antennas record time-resolved amplitudes of radio emission
from a small region of the sky \cite{perley1989synthesis}.
The correlation between the amplitudes at any pair of antennas defines a sampling point
in the spatial Fourier representation of the image;
its position is determined by the vectorial distance between the antennas.
This sparse frequency sampling lends itself to reconstruction approches based on compressed sensing
\cite{candes2006compressive,wiaux2009compresseda,wiaux2009compressed,wenger2010compressed,wenger2010sparseri,wenger2013group}.

A single, ``snapshot'' measurement results in a characteristic Fourier domain sampling pattern, \ref{fig:ri-sampling}.
Since this pattern is very sparse, it often contains too little information for accurate reconstruction.
However, over the course of a day, the sampling pattern changes with the rotation of the Earth.
Thus, by combining the data from several measurements, much better coverage can be obtained, \ref{fig:ri-longsampling}.

\begin{figure}[ht]
  \begin{center}
    \begin{subfigure}{0.3\linewidth}
      \includegraphics[width=\textwidth]{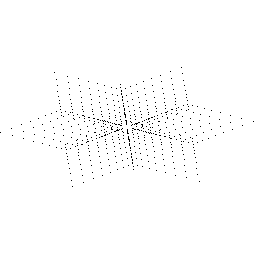}
      \caption{Single snapshot.}
      \label{fig:ri-sampling}
    \end{subfigure}\hspace{0.02\linewidth}%
    \begin{subfigure}{0.3\linewidth}
      \includegraphics[width=\textwidth]{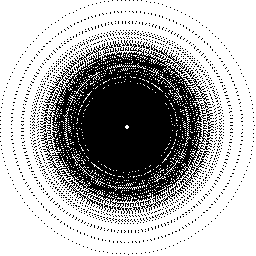}
      \caption{Half-day observation.}
      \label{fig:ri-longsampling}
    \end{subfigure}
  \end{center}
	\caption{Interferometric sampling patterns.}
\end{figure}

Since observation time is scarce and expensive,
it is desirable to only collect the minimum number of measurements necessary for reconstruction.
Using an online reconstruction algorithm, \ref{sec:online-cs},
it is possible to continuously monitor the reconstruction result while new data is added.
More importantly, the residual $\norm{A x^k - b}$
provides feedback about the amount of information in each incoming data block:
a new block that contains additional information instantly increases the residual because of the additional rows in $A$,
while a redundant block causes no significant change in the residual.
The absence of this increase in the residual value can serve as a stopping criterion for the measurement process.

\begin{figure}[ht]
  \begin{center}
    \begin{subfigure}{0.3\linewidth}
      \includegraphics[width=\textwidth]{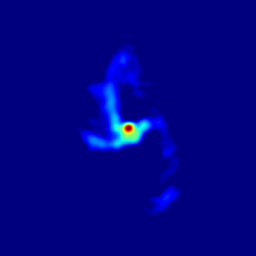}
      \caption{Sagittarius~A West.}
      \label{fig:ri-orig}
    \end{subfigure}\hspace{0.02\linewidth}%
    \begin{subfigure}{0.3\linewidth}
      \includegraphics[width=\textwidth]{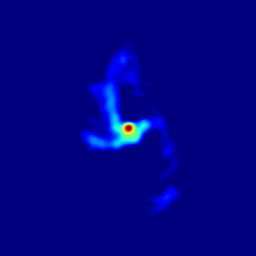}
      \caption{Reconstruction.}
      \label{fig:ri-rec}
    \end{subfigure}
  \end{center}
	\caption{Radio source and reconstruction from simulated data.}
\end{figure}

We illustrate this application using simulated measurements of the radio source Sagittarius~A West, \ref{fig:ri-orig},
in 7.5-minute intervals with the Very Large Array telescope.
The sampling pattern of each 7.5-minute observational data block resembles a rotated copy of \ref{fig:ri-sampling};
after a 12-hour observation, this results in the pattern shown in \ref{fig:ri-longsampling}.
Each data block $b_l$ is related to the true image by a corresponding sampled Fourier transform $A_l$.
While waiting for the next block, 300 iterations of \eqref{eq:iterations} are performed on the blocks accumulated so far.
Whenever a new data block is added, $r(k)$ starts with the newest block, followed by all others in a cyclic scheme.
$\lambda$ is set to $10^{-4} \norm{x}_1$, where $x$ is the ground truth image. 
Since negative intensities are not physically possible, we enforce the additional constraint $x \geq 0$
by truncating negative components of $z^{k+1}$ to zero before evaluating $S_\lambda(z^{k+1})$
(a step that amounts to a Bregman projection onto the non-negative orthant, cf.~\cite[Lemma 2.5]{lorenz2013linbreg}).

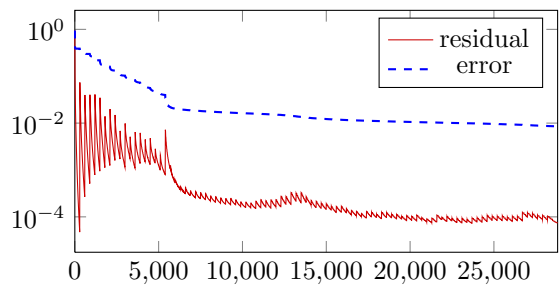
\begin{figure}[ht]
  \begin{center}
    \begin{tikzpicture}
      \begin{axis}[ width=0.5\linewidth, height=0.3\linewidth,
        ymode=log, xmin=0, xmax=28800, scaled ticks=false, legend
        entries={residual,error}, legend pos = north east ]
        \addplot[style=solid,red!80!black] table[x=step,y=residual]
        {data/ri/ri-sgra-10.dat}; \addplot[style=dashed,blue,thick]
        table[x=step,y=error] {data/ri/ri-sgra-10.dat};
      \end{axis}
    \end{tikzpicture}
  \end{center}
	\caption{Relative residuals $\norm{A x^k - b } / \norm{b}$ and errors
	of online interferometric reconstruction after each time step $k$.}
	\label{fig:ri-res}
\end{figure}

The residuals after each step are plotted in Figure~\ref{fig:ri-res}.
After 5\,400 iterations, when 18 blocks have been processed,
additional data blocks stop causing a significant increase in the residual,
and the relative reconstruction error abruptly drops from about 4\,\% to about 2\,\%.
In a practical setting, the measurement could now be stopped, based on an observation that relies only on the residual, freeing the telescope for other observations,
and the algorithm could perform additional iterations on the available data until convergence.
In our experiment, we successively add the remaining blocks,
and the algorithm converges to a final relative error of about 0.85\,\%, \ref{fig:ri-rec}. 

\bibliographystyle{IEEEbib}
\bibliography{refs}
\bigskip

\noindent
Dirk Lorenz\\
\texttt{d.lorenz@tu-braunschweig.de}, Institute for Analysis and Algebra, TU Braunschweig, 38092 Braunschweig, Germany\\

\noindent
Stephan Wenger,  Marcus Magnor\\
\texttt{\{wenger,magnor\}@cg.cs.tu-bs.de},
Institut für Computergraphik, TU Braunschweig, 38092 Braunschweig, Germany\\

\noindent
Frank Sch\"opfer\\
\texttt{frank.schoepfer@uni-oldenburg.de}, Carl von Ossietzky Universit\"at Oldenburg, 26111 Oldenburg, Germany

\end{document}